\pdfoutput=1
\RequirePackage{ifpdf}
\ifpdf 
\documentclass[pdftex]{sigma}
\else
\documentclass{sigma}
\fi

\begin{document}

\allowdisplaybreaks

\renewcommand{\thefootnote}{$\star$}

\newcommand{\arXivNumber}{1510.08599}

\renewcommand{\PaperNumber}{042}

\FirstPageHeading

\ShortArticleName{Zeros of Quasi-Orthogonal Jacobi Polynomials}

\ArticleName{Zeros of Quasi-Orthogonal Jacobi Polynomials\footnote{This paper is a~contribution to the Special Issue
on Orthogonal Polynomials, Special Functions and Applications.
The full collection is available at \href{http://www.emis.de/journals/SIGMA/OPSFA2015.html}{http://www.emis.de/journals/SIGMA/OPSFA2015.html}}}

\Author{Kathy DRIVER and Kerstin JORDAAN}
\AuthorNameForHeading{K.~Driver and K.~Jordaan}
\Address{Department of Mathematics and Applied Mathematics, University of Pretoria,\\ Pretoria, 0002, South Africa}
\Email{\href{mailto:kathy.driver@uct.ac.za}{kathy.driver@uct.ac.za}, \href{mailto:kerstin.jordaan@up.ac.za}{kerstin.jordaan@up.ac.za}}

\ArticleDates{Received October 30, 2015, in f\/inal form April 20, 2016; Published online April 27, 2016}

\Abstract{We consider interlacing properties satisf\/ied by the zeros of Jacobi polynomials in quasi-orthogonal sequences characterised by $\alpha>-1$, $-2<\beta<-1$. We give necessary and suf\/f\/icient conditions under which a conjecture by Askey, that the zeros of Jacobi polyno\-mials~$P_n^{(\alpha, \beta)}$ and~$P_{n}^{(\alpha,\beta+2)}$ are interlacing, holds when the parameters~$\alpha$ and~$\beta$ are in the range $\alpha>-1$ and $-2<\beta<-1$. We prove that the zeros of $P_n^{(\alpha, \beta)}$ and $P_{n+1}^{(\alpha,\beta)}$ do not interlace for any $n\in\mathbb{N}$, $n\geq2$ and any f\/ixed~$\alpha$,~$\beta$ with $\alpha>-1$, $-2<\beta<-1$. The interlacing of zeros of~$P_n^{(\alpha,\beta)}$ and $P_m^{(\alpha,\beta+t)}$ for $m,n\in\mathbb{N}$ is discussed for~$\alpha$ and~$\beta$ in this range, $t\geq 1$, and new upper and lower bounds are derived for the zero of $P_n^{(\alpha,\beta)}$ that is less than~$-1$.}

\Keywords{interlacing of zeros; quasi-orthogonal Jacobi polynomials}

\Classification{33C50; 42C05}

\renewcommand{\thefootnote}{\arabic{footnote}}
\setcounter{footnote}{0}

\section{Introduction}
Let $\{p_n\}_{n=0}^\infty$, $\deg(p_n)=n$, $n\in\mathbb{N}$, be a sequence of orthogonal polynomials with respect to a~positive Borel measure $\mu$ supported on an interval~$(a,b)$. It is well known
(see~\cite{Sze}) that the zeros of $p_n$ are real and simple and lie in $(a,b)$ while, if we denote the zeros of $p_n$, in increasing
order by $x_{1,n}<x_{2,n}<\dots<x_{n,n}$, then
\begin{gather*}
x_{1,n}<x_{1,n-1}<x_{2,n}<x_{2,n-1}<\dots<x_{n-1,n-1}<x_{n,n},
\end{gather*}
a property called the interlacing of zeros.

Since our discussion will include interlacing of zeros of polynomials of non-consecutive degree, we recall the following def\/initions:
\begin{definition}\label{SI}
Let $n \in {\mathbb N}$. If $x_{1,n}<x_{2,n}<\dots<x_{n,n}$ are the zeros of $p_n$ and
$y_{1,n}<y_{2,n}<\dots<y_{n,n}$ are the zeros of $q_{n}$, then the zeros of $p_n$ and $q_{n}$ are interlacing if
\begin{gather*}x_{1,n}<y_{1,n}<x_{2,n}<y_{2,n}<\dots<x_{n,n}<y_{n,n}
\end{gather*}
or if
\begin{gather*}y_{1,n}<x_{1,n}<y_{2,n}<x_{2,n}<\dots<y_{n,n}<x_{n,n},
\label{1.5}\end{gather*}
\end{definition}

The def\/inition of interlacing of zeros of two polynomials whose degrees dif\/fer by more than one was introduced by Stieltjes~\cite{Sze}.

\begin{definition}
 Let $m,n \in {\mathbb N}$, $m \leq n-2$. The zeros of the polynomials $p_n$ and~$q_m$ are interlacing if there exist~$m$ open intervals, with endpoints at successive zeros of~$p_n$, each of which contains exactly one zero of~$q_m$.
 \end{definition}

{\bf Askey conjecture.} In \cite{Ask}, Richard Askey
conjectured that the zeros of the Jacobi polyno\-mials~$P_{n}^{(\alpha,\beta)}$ and $P_n^{(\alpha,\beta+2)}$ are interlacing for each $n \in {\mathbb N}$, $\alpha,\beta > -1$. A more
general version of the Askey conjecture was proved in~\cite{DrJoMb}, namely that
the zeros of $P_{n}^{(\alpha,\beta)}$ and the zeros of $P_n^{(\alpha-k,\beta+t)}$ are interlacing
for each $n\in\mathbb{N}$, $\alpha,\beta>-1$ and any real numbers $t$ and~$k$ with $0\leq t,k\leq 2$.

Here, we investigate Askey's conjecture,
and several extensions thereof, in the context of sequences of Jacobi polynomials that are
quasi-orthogonal of order~$1$.

The concept of quasi-orthogonality of order $1$ was introduced by
Riesz in~\cite{Rie} in his seminal work on the moment problem. Fej\'{e}r~\cite{Fej} considered
quasi-orthogonality of order $2$ while the general case was f\/irst studied by Shohat~\cite{Sho}. Chihara~\cite{Chi} discussed quasi-orthogonality of order $r$ in the context of three-term recurrence relations and Dickinson~\cite{Dickinson} improved Chihara's result by deriving a system of recurrence relations that provides both necessary and suf\/f\/icient conditions for quasi-orthogonality. Algebraic properties of the linear functional associated with quasi-orthogonality are investigated in \cite{Draux, maroni1, maroni2, maroni}. Quasi-orthogonal polynomials have also been studied in the context of connection coef\/f\/icients, see for example \cite{area, askey65,Dimitrov, Ism, sz,tr73,tr75,wilson} as well as Geronimus canonical spectral transformations of the measure (cf.~\cite{zh}). Properties of orthogonal polynomials associated with such Geronimus perturbations, including properties satisf\/ied by the zeros, have been analysed in~\cite{Branq}.

The def\/inition of quasi-orthogonality of a sequence of polynomials is the following:

\begin{definition}
 Let $\{q_n\}_{n=0}^\infty$ be a sequence of polynomials with degree $q_n = n$ for each $n \in
 {\mathbb N}$. For a positive integer $r < n$, the sequence $\{q_n\}_{n=0}^\infty$ is
 quasi-orthogonal of order $r$ with respect to a positive Borel measure $\mu$ if
\begin{gather}
\int x^k q_n(x) d{\mu} (x) =0 \qquad \mbox{for} \quad k=0,\ldots,n-1-r.
 \label{2}
 \end{gather}
\end{definition}

If (\ref{2}) holds for $r=0$, the sequence $\{q_n\}_{n=0}^\infty$ is orthogonal with respect to
 the measure~$\mu$. A~cha\-racterisation of a polynomial $q_n$ that is quasi-orthogonal of order~$r$ with respect to a~positive measure~$\mu$, as a linear combination of $p_n,p_{n-1},\dots,p_{n-r}$ where $\{p_n\}_{n=0}^\infty$ is orthogonal with respect to~$\mu$, was f\/irst investigated by Shohat (cf.~\cite{Sho}). A full statement and proof of this result can be found in \cite[Theorem~1]{BDR}.

Quasi-orthogonal polynomials arise in a natural way in the context of
 classical orthogonal polynomials that depend on one or more parameters. The sequence of Jacobi polyno\-mials~$\big\{P_{n}^{(\alpha, \beta)}\big\} _{n=0}^\infty$ is orthogonal on $(-1,1)$ with respect to the weight function $(1-x)^{\alpha}(1+x)^{\beta}$ when $\alpha > -1$, $\beta > -1$. The three-term recurrence relation \cite[(4.5.1)]{Sze} satisf\/ied by the sequence is
 \begin{gather}
 c_n P_n^{(\alpha, \beta)} (x) = (x - d_n) P_{n-1}^{(\alpha, \beta)} (x) - e_n P_{n-2}^{(\alpha, \beta)} (x), \qquad n = 2,3,\dots, \label{1}
 \end{gather}
where
\begin{gather}c_{n} = 2n(n +\alpha + \beta)/ (2n +\alpha + \beta-1)(2n +\alpha + \beta),\nonumber\\
 d_{n} = \big({\beta}^2 - {\alpha}^2\big)/(2n +\alpha + \beta -2) (2n +\alpha + \beta), \label{d}\\
 e_{n} = 2(n +\alpha -1) (n + \beta -1)/(2n +\alpha + \beta-2)(2n +\alpha + \beta -1),\nonumber
 \end{gather}
and $P_{0}^{(\alpha,\beta)} (x) \equiv 1$, $P_1^{(\alpha,\beta)} (x) = \frac12 (\alpha
+\beta +2) x +\frac12(\alpha - \beta)$. For values of $\alpha$ and $\beta$ outside the range
$\alpha ,\beta > -1$, the Jacobi sequence $\big\{P_{n}^{(\alpha, \beta)}\big\} _{n=0}^\infty$ can be
def\/ined by the three term recurrence relation~\eqref{1}. The quasi-orthogonal Jacobi sequences of order $1$ and $2$ are of particular interest since, apart from the orthogonal Jacobi sequences, these are the only sequences of Jacobi polynomials for $\alpha,\beta\in \mathbb{R}$ where all~$n$ zeros of $P_{n}^{(\alpha,\beta)}$ are real and distinct.

In \cite[Theorem~7]{BDR}, it is proved that if $-1< \alpha,\beta <0$, and $k,l \in {\mathbb N}$ with $k+l<n$, the
 Jacobi polynomials $\big\{P_{n}^{(\alpha-k, \beta-l)}\big\} _{n=0}^\infty$ are quasi-orthogonal of order~$k+l$ with respect to the weight function $(1-x)^{\alpha}(1+x)^{\beta}$ on the interval $[-1,1]$.

Interlacing properties of zeros of quasi-orthogonal and orthogonal Jacobi polynomials of the same or consecutive degree were discussed and the following result proved in \cite[Corollary~4]{BDR}.

\begin{lemma}\label{bdr}
Fix $\alpha$ and $\beta$, $\alpha > -1$ and $-2 < \beta< -1$ and denote the sequence of Jacobi
polynomials by $\big\{P_{n}^{(\alpha, \beta)}\big\}_{n=0}^\infty$. For each $n \in{\mathbb N}$, $n \geq
1$, let $x_{1,n}<x_{2,n}<\dots<x_{n,n}$ denote the zeros of the $($quasi-orthogonal$)$ polynomial $P_{n}^{(\alpha,\beta)}$ and
$y_{1,n}<y_{2,n}<\dots<y_{n,n}$ denote the zeros of the $($orthogonal$)$ polynomial $P_{n}^{(\alpha,\beta +1)}$. Then
\begin{gather} x_{1,n}< -1 < y_{1,n} < x_{2,n}< y_{2,n} <\dots< x_{n,n}< y_{n,n}<1
\label{2.1}
\end{gather}
and
\begin{gather}x_{1,n+1}< -1 < y_{1,n} < x_{2,n+1}< y_{2,n} <\dots< x_{n,n+1}< y_{n,n}<
x_{n+1,n+1}<1, \label{2.2}
\end{gather}
\end{lemma}

For proof see \cite[Corollary~4(ii)(a)]{BDR} with $\beta$ replaced by $\beta +1$.

In \cite{bu}, Bustamante, Mart\'{i}nez-Cruz and Quesada apply the interlacing properties of zeros of quasi-orthogonal and orthogonal Jacobi polynomials given in~\cite{DrJoMb} and in Lemma~\ref{bdr} to show that best possible one-sided polynomial approximants to a unit step function on the interval $[-1,1]$, which are in some cases unique, can be obtained using Hermite interpolation at interlaced zeros of quasi-orthogonal and orthogonal Jacobi polynomials.

We assume throughout this paper that~$\alpha$ and $\beta$ are f\/ixed numbers lying in the range $\alpha>-1$, $-2<\beta<-1$.

In Section~\ref{sectionintqo}, we analyse the interlacing properties of zeros of polynomials of consecutive, and non-consecutive, degree within a sequence of quasi-orthogonal Jacobi polynomials of order~$1$. In Section~\ref{sectionAskey} we prove a necessary and suf\/f\/icient condition for the Askey conjecture to hold between the zeros of an orthogonal and a quasi-orthogonal (order~1) sequence of Jacobi polynomials of the same degree and then extend this to the case where the polynomials are of consecutive degree. In Section~\ref{sectionst} we discuss interlacing properties and inequalities satisf\/ied by the zeros of orthogonal and quasi-orthogonal (order~$1$) Jacobi polynomials whose degrees dif\/fer by more than unity. In Section~\ref{bounds} we derive upper and lower bounds for the zero of $P_n^{(\alpha,\beta)}$ that is $<-1$.

Note that, since Jacobi polynomials satisfy the symmetry property \cite[equation~(4.1.1)]{Ism}
\begin{gather}
P_n^{(\alpha, \beta)} (x) = {(-1)}^n P_n^{(\beta, \alpha)} (-x), \label{2.6}
\end{gather}
each result proved for quasi-orthogonal Jacobi polynomials $P_n^{(\alpha,\beta)}$ with $\alpha > -1$, $-2 < \beta< -1$ has an analogue for the corresponding quasi-orthogonal polynomial with $\beta>-1$, $-2<\alpha<-1$.

\section{Quasi-orthogonal Jacobi polynomials of order 1}\label{sectionintqo}

\subsection[Zeros of $P_{n}^{(\alpha,\beta)}$ and $P_{n-k}^{(\alpha,\beta)}$, $k,n\in\mathbb{N}$, $1\leq k<n$]{Zeros of $\boldsymbol{P_{n}^{(\alpha,\beta)}}$ and $\boldsymbol{P_{n-k}^{(\alpha,\beta)}}$, $\boldsymbol{k,n\in\mathbb{N}}$, $\boldsymbol{1\leq k<n}$}

Our f\/irst result proves that for any $n \in {\mathbb N}$, $n \geq 2$, interlacing does not hold between the
$n$ real zeros of $P_{n}^{(\alpha,\beta)}$ and the $n+1$ real zeros of $P_{n+1}^{(\alpha,\beta)}$.
However, the $n-1$ zeros of $P_{n}^{(\alpha,\beta)}$ in $(-1,1)$ interlace with the $n$ zeros of $P_{n+1}^{(\alpha,\beta)}$ in $(-1,1)$.
Moreover, the $n+1$ zeros of $(1+x) P_{n}^{(\alpha,\beta)}(x)$ interlace with the $n+1$ zeros of
$P_{n+1}^{(\alpha,\beta)}(x)$ for each $n\in\mathbb{N}$.

\begin{theorem} \label{Th:2.1}
Fix $\alpha$ and $\beta$, $\alpha > -1$ and $-2 < \beta< -1$ and denote the sequence of Jacobi
polynomials by $\big\{P_{n}^{(\alpha, \beta)}\big\}_{n=0}^\infty$. For each $n \in{\mathbb N}$, $n \geq
1$, let $x_{1,n}<x_{2,n}<\dots<x_{n,n}$ denote the zeros of $P_{n}^{(\alpha,\beta)}$. Then
\begin{gather}x_{1,n} < x_{1,n+1}< -1 < x_{2,n+1}< x_{2,n} <\dots < x_{n,n+1}< x_{n,n}< x_{n+1,n+1}<1. \label{2.3}
\end{gather}
\end{theorem}

\begin{corollary} Let $\big\{P_{n}^{(\alpha, \beta)}\big\}_{n=0}^\infty$ denote the sequence of Jacobi polynomials and fix~$\alpha$ and~$\beta$ with $\alpha > -1$ and $-2 < \beta< -1$.
The zeros of $P_{n-k}^{(\alpha,\beta)}$ and the zeros of $P_{n}^{(\alpha,\beta)}$ do not interlace for any $k,n \in{\mathbb N}$, $n \geq 3$, $k\in\{1,\dots,n-1\}$.
\end{corollary}

\begin{proof}
It follows immediately from Def\/inition \ref{SI} that Stieltjes interlacing
 does not hold between the zeros of two polynomials if any zero of the polynomial of smaller
 degree lies outside the interval with endpoints at the smallest and largest zero of the
 polynomial of larger degree. Since~\eqref{2.3} shows that $x_{1,n-2} < x_{1,n-1} < x_{1,n} <
 -1 < x_{n,n}$ for each $n \in{\mathbb N}$, the smallest zero of~$P_{n-2}^{(\alpha,\beta)}$ lies outside the interval $(x_{1,n},x_{n,n})$ and this proves the result.
\end{proof}

\begin{remark}
Theorem \ref{Th:2.1} complements results proved by Dimitrov, Ismail and Rafaeli~\cite{DiIsRa} who consider the interlacing properties of zeros of orthogonal polynomials arising from perturbations of the weight function of orthogonality. However, the sequences of polynomials considered in \cite{DiIsRa} retain orthogonality. Shifting from the orthogonal case to the quasi-orthogonal order~$1$ Jacobi case $P_n^{(\alpha,\beta)}$ with $\alpha>-1$, $-2<\beta<-1$ may be viewed as a perturbation of the (orthogonal) Jacobi weight function ${(1-x)}^{\alpha}{(1+x)}^{\beta}$, $\alpha,\beta>-1$, by the factor ${(1+x)}^{-1}$.
\end{remark}

\begin{remark}\label{rem} Relation~\eqref{2.3} proves that for each f\/ixed $\alpha$ and $\beta$ with $\alpha >-1$ and $-2 < \beta< -1$, the zero of~$P_{n}^{(\alpha,\beta)}$ that is less than $-1$, increases with~$n$.
\end{remark}

\begin{corollary} \label{cor:2.2}
For each fixed $\alpha$, $\beta$ with $-2 < \alpha < -1$ and $\beta > -1$, and each $n \in
{\mathbb N}$, $n \geq 2$,
\begin{itemize}\itemsep=0pt
\item[$(i)$]the $n+1$ zeros of $(1-x) P_{n}^{(\alpha,\beta)}(x)$ interlace with the
$n+1$ zeros of $P_{n+1}^{(\alpha,\beta)}(x)$;
\item[$(ii)$] the $n-1$ zeros of $P_{n}^{(\alpha,\beta)}$ that
lie in the interval $(-1,1)$ interlace with the $n$ zeros of $P_{n+1}^{(\alpha,\beta)}$ that lie
in the interval $(-1,1)$;
\item[$(iii)$] interlacing does not hold between all the real zeros of $P_{n}^{(\alpha,\beta)}$ and all the real zeros of
$P_{n+1}^{(\alpha,\beta)}$ for any $n \in {\mathbb N}$, $n \geq 2$;
\item[$(iv)$] the zero of $P_n^{(\alpha,\beta)}$ that is $>1$ decreases with~$n$.
\end{itemize}
\end{corollary}

\begin{proof} The result follows from Theorem~\ref{Th:2.1} and the symmetry property (\ref{2.6}) of Jacobi polyno\-mials.
\end{proof}

\subsection[Co-primality and zeros of $P_{n}^{(\alpha,\beta)}$ and $P_{n-k}^{(\alpha,\beta)}$, $k,n\in\mathbb{N}$, $2\leq k<n$]{Co-primality and zeros of $\boldsymbol{P_{n}^{(\alpha,\beta)}}$ and $\boldsymbol{P_{n-k}^{(\alpha,\beta)}}$, $\boldsymbol{k,n\in\mathbb{N}}$, $\boldsymbol{2\leq k<n}$}

Common zeros of two polynomials, should they exist,
play a crucial role when discussing interlacing properties of their zeros, see, for example, \cite{DrMu}. The polynomials $P_{n}^{(\alpha,\beta)}$ and $P_{n-1}^{(\alpha,\beta)}$ of consecutive degree are co-prime for each $n \in{\mathbb N}$, $n \geq 1$, and each f\/ixed~$\alpha$,~$\beta$ with $\alpha > -1$ and $-2
< \beta< -1$. This follows from Theorem~\ref{Th:2.1} but is also immediate from the three term recurrence
relation~\eqref{1} since if $P_{n}^{(\alpha,\beta)}$ and $P_{n-1}^{(\alpha,\beta)}$ had a common
zero, this would also be a zero of $P_{n-2}^{(\alpha,\beta)}$. After suitable
iteration of~(\ref{1}), this contradicts $P_{0}^{(\alpha,\beta)} (x) \equiv 1$.
\begin{theorem}\label{Th:3.1} Let $\big\{P_{n}^{(\alpha, \beta)}\big\}_{n=0}^\infty$ denote the sequence of Jacobi polynomials and fix~$\alpha$ and $\beta$ with $\alpha > -1$ and $-2 < \beta< -1$. If $P_{n}^{(\alpha,\beta)}$ and
 $P_{n-2}^{(\alpha,\beta)}$ are co-prime for each $n \in{\mathbb N}$, $n \geq 3$, then the zeros of $(x+1)(x-d_n)P_{n-2}^{(\alpha,\beta)}$ interlace with the zeros of $P_{n}^{(\alpha,\beta)}$ where $d_n$ is given in~\eqref{d}.
\end{theorem}

\begin{remark}We note that results analogous to Theorem~\ref{Th:3.1} can be proven for the zeros of Jacobi polynomials $P_{n}^{(\alpha,\beta)}$ and
 $P_{n-k}^{(\alpha,\beta)}$ when $k,n\in\mathbb{N}$, $3\leq k < n$.
 \end{remark}

\section{An extension of the Askey conjecture}\label{sectionAskey}

\subsection[Zeros of $P_{n}^{(\alpha,\beta)}$ and $P_{n}^{(\alpha,\beta+2)}$, $n\in\mathbb{N}$]{Zeros of $\boldsymbol{P_{n}^{(\alpha,\beta)}}$ and $\boldsymbol{P_{n}^{(\alpha,\beta+2)}}$, $\boldsymbol{n\in\mathbb{N}}$}

We investigate an extension of the Askey conjecture that the zeros of the Jacobi polyno\-mials~$P_{n}^{(\alpha,\beta)}$ and $P_{n}^{(\alpha,\beta+2)}$ are interlacing when $\alpha > -1$, $-2 < \beta< -1$ and prove a necessary and suf\/f\/icient condition for interlacing between the zeros of these two polynomials to occur.

\begin{theorem}\label{Th:Askey1}
Suppose that $\alpha > -1$, $-2 < \beta< -1$, and $\big\{P_n^{(\alpha,\beta)}\big\}_{n=0}^\infty$ is the sequence of Jacobi polynomials. Let $\delta:= -1-\frac{2(\beta+1)}{\alpha +\beta+2n +2}$. For each $n \in{\mathbb N}$, the zeros of $P_{n}^{(\alpha,\beta)}$ and $P_{n}^{(\alpha,\beta+2)}$ are interlacing if and only if $\delta < x_{2,n}$, where $x_{2,n}$ is the smallest zero of~$P_n^{(\alpha,\beta)}$ in the interval $(-1,1)$.
\end{theorem}

\begin{remark} Numerical evidence conf\/irms that the assumption in Theorem~\ref{Th:Askey1}, i.e., $\delta<x_{2,n}$, is reasonable. There are values of $\alpha$ and $\beta$ for which the condition is satisf\/ied and others where it is not. For example, when $n=5$, $\alpha=2.35$ and $\beta =-1.5$ we have $\delta=-0.922179$ and $x_{2,n}=-0.885666$ whereas for the same $n$ and $\alpha$ with $\beta =-1.9$ we have $\delta=-0.855422$ and $x_{2,n}=-0.961637$. Analytically one can see that the condition is more likely to be satisf\/ied when $\delta$ approaches~$-1$, a lower bound for~$x_{2,n}$, i.e., when $\beta \to -1$ with $\alpha>-1$ and $n\in\mathbb{N}$ f\/ixed.
\end{remark}

Although full interlacing between the zeros of $P_n^{(\alpha,\beta)}$ and $P_n^{(\alpha,\beta+2)}$ cannot occur when $\delta \geq x_{2,n}$, there is an interlacing result, involving the point~$\delta$, that holds between the zeros of $P_n^{(\alpha,\beta)}$ that lie in the interval $(-1,1)$ and the zeros of $P_n^{(\alpha,\beta+2)}$ provided the two polynomials have no common zeros.
\begin{theorem}\label{Th:Askey2}
Suppose that $\alpha > -1$, $-2 < \beta < -1$ and $\big\{P_n^{(\alpha,\beta)}\big\}_{n=0}^\infty$ is the sequence of Jacobi polynomials. Suppose that $P_n^{(\alpha,\beta)}$ and $P_n^{(\alpha,\beta+2)}$ have no common zeros and assume that $\delta:= -1-\frac{2(\beta+1)}{\alpha +\beta+2n +2}> x_{2,n}$. For each $n \in{\mathbb N}$, the zeros of $(x-\delta) P_n^{(\alpha,\beta)}(x)$ interlace with the zeros of $P_n^{(\alpha,\beta+2)}(x)$.
\end{theorem}

\subsection[Zeros of $P_{n}^{(\alpha,\beta)}$ and $P_{n-1}^{(\alpha,\beta+2)}$, $n\in\mathbb{N}$]{Zeros of $\boldsymbol{P_{n}^{(\alpha,\beta)}}$ and $\boldsymbol{P_{n-1}^{(\alpha,\beta+2)}}$, $\boldsymbol{n\in\mathbb{N}}$}

\begin{theorem}\label{Th:la} Let $\alpha > -1$, $-2 < \beta < -1$ and $\big\{P_n^{(\alpha,\beta)}\big\}_{n=0}^\infty$ be the sequence of Jacobi polyno\-mials. Let $x_{1,n}<x_{2,n}<\dots<x_{n,n}$ denote the zeros of $P_{n}^{(\alpha,\beta)}$ and $z_{1,n-1}<z_{2,n-1}<\dots<z_{n-1,n-1}$ denote the zeros of $P_{n-1}^{(\alpha,\beta+2)}$.
Then $P_{n-1}^{(\alpha,\beta+2)}$ and $P_n^{(\alpha,\beta)}$ are co-prime and the zeros of \mbox{$(1+x)P_{n-1}^{(\alpha, \beta+2)}$} interlace with the zeros of $P_n^{(\alpha,\beta)}$, i.e.,
\begin{gather*}
x_{1,n}<-1<x_{2,n}<z_{1,n-1}<x_{3,n}<\dots<x_{n-1,n}<z_{n-2,n-1}<x_{n,n}<z_{n-1,n-1}.
\end{gather*}
\end{theorem}

\section[Zeros of $P_n^{(\alpha,\beta)}$ and $P_{n-2}^{(\alpha,\beta+t)}$, $t \geq 1$, $n\in\mathbb{N}$]{Zeros of $\boldsymbol{P_n^{(\alpha,\beta)}}$ and $\boldsymbol{P_{n-2}^{(\alpha,\beta+t)}}$, $\boldsymbol{t \geq 1}$, $\boldsymbol{n\in\mathbb{N}}$}\label{sectionst}

For f\/ixed $\alpha>-1$, $-2 < \beta< -1$, and f\/ixed $t \geq 1$, the parameter $\beta+t$ is
greater than $-1$ and each sequence of Jacobi polynomials $\big\{P_{n}^{(\alpha,\beta+t)}\big\} _{n=0}^\infty$ is orthogonal on the interval $(-1,1)$. It is known (see \eqref{2.1} and \eqref{2.2}) that the zeros of the quasi-orthogonal polynomial $P_n^{(\alpha,\beta)}$ interlace with the zeros of
the (orthogonal) polynomial $P_{n-1}^{(\alpha,\beta+1)}$, as well as with the zeros of the (orthogonal) polynomial
$P_n^{(\alpha,\beta+1)}$. Here, we discuss interlacing between the zeros of~$P_n^{(\alpha,\beta)}$
and the zeros of the (orthogonal) polynomial $P_{n-2}^{(\alpha, \beta+1)}$. We also prove that
the zeros of~$P_n^{(\alpha,\beta)}$ and the zeros of $P_{n-2}^{(\alpha,\beta+t)}$ interlace for
continuous variation of $t$, $2 \leq t \leq 4$ and that the polyno\-mials~$P_n^{(\alpha,\beta)}$ and~$ P_{n-2}^{(\alpha,\beta+t)}$ are co-prime for any $t\in[2,4]$

\begin{theorem}\label{Th:St1} Let $n \in{\mathbb N}$, $n\geq3$, $\alpha$, $\beta$ fixed, $\alpha >-1$, $-2 < \beta < -1$, and suppose $\big\{P_n^{(\alpha,\beta)}\big\}_{n=0}^\infty$ is the sequence of Jacobi polynomials.
\begin{itemize}\itemsep=0pt
\item[$(i)$] The $n-2$ distinct zeros of $P_{n-2}^{(\alpha,\beta+1)}$ $($which all lie in the
 interval $(-1,1))$ together with the point $\frac{2(n+\beta)(\alpha+\beta+n)}{(\alpha+\beta+2n)(\alpha+\beta+2n-1)}-1$, interlace
with the $n-1$ distinct zeros of $P_n^{(\alpha,\beta)}$ that lie in~$(-1,1)$, provided $P_{n-2}^{(\alpha,\beta+1)}$ and $P_n^{(\alpha,\beta)}$ are co-prime.
\item[$(ii)$] For $2 \le t \le 4$, the $n-2$ distinct zeros of $P_{n-2}^{(\alpha,\beta+t)}$
 interlace with the $n-1$ zeros of $P_n^{(\alpha,\beta)}$ that lie in $(-1,1)$.
\end{itemize}
\end{theorem}

\begin{remark}
Note that Theorem \ref{Th:St1}(ii) does not assume that $P_{n-2}^{(\alpha,\beta+t)}$ and
$P_n^{(\alpha,\beta)}$ are co-prime, $2 \le t \le 4$. This assumption is not required since the proof will show that $P_{n-2}^{(\alpha,\beta+t)}$ and $P_n^{(\alpha,\beta)}$ are co-prime for every~$t$, $2 \le t \le 4$, $\alpha >-1$, $-2 < \beta < -1$, and $n \in{\mathbb N}$.
\end{remark}

\section[Bounds for the smallest zeros of $P_n^{(\alpha,\beta)}$, $n\in\mathbb{N}$]{Bounds for the smallest zeros of $\boldsymbol{P_n^{(\alpha,\beta)}}$, $\boldsymbol{n\in\mathbb{N}}$}\label{bounds}

In this section we derive upper and lower bounds for the zero
of $P_{n}^{(\alpha,\beta)}$ that lies outside the interval $(-1,1)$ when $\alpha > -1$, $-2 < \beta< -1$.
\begin{theorem}\label{Th:bounds} Let $n \in{\mathbb N}$, $n\geq3$, $\alpha$, $\beta$ fixed, $\alpha >-1$, $-2 < \beta < -1$. Denote the smallest zero of the Jacobi polynomial $P_n^{(\alpha,\beta)}$ by $x_{1,n}$. Then
\begin{gather}\label{bd}
-1+A_n<-1+\frac{D_n}{C_n}<x_{1,n}<-B_n<-1,
\end{gather}
where \begin{subequations}\label{labels}
\begin{gather}
A_n=\frac{2(\beta+1)}{2n+\alpha+\beta},\\
B_n=1-\frac{2(\beta +1)(\beta +2)}{(n+\beta+1) (n+\alpha +\beta+1)},\\
C_n=(\beta +3) (\alpha +\beta +2)+2 (n-1) (n+\alpha +\beta+2),\\
D_n=2(\beta+1)(\beta+3).
\end{gather}\end{subequations}
\end{theorem}

The upper and lower bounds obtained in Theorem \ref{Th:bounds} for the zero of $P_n^{(\alpha,\beta)}$, $\alpha>-1$, $-2<\beta<-1$, that is smaller than $-1$, approach $-1$ as $n\to\infty$. This is consistent with the observation that this zero increases with~$n$ (cf.\ Remark~\ref{rem}). These bounds for the smallest zero of a~quasi-orthogonal (order~1) Jacobi polynomial are remarkably good. We provide some numerical examples in Table~\ref{Jac} to illustrate the inequalities in~\eqref{bd}.

\begin{table}[!ht] \centering \caption{Bounds for the smallest zero of $P_{15}^{(\alpha,\beta)}(x)$ for dif\/ferent values of $\alpha>-1$ and $-2<\beta<-1$.}\label{Jac}
\vspace{1mm}
\begin{tabular}{|c|c|c|c|}
\hline
$\alpha$, $\beta$ & $-1+\frac{D_n}{C_n}$ & $x_{1,15}$ & $-B_n$\tsep{3pt}\bsep{3pt}\\
\hline
$\alpha=0.93$, $\beta=-1.9$ & $-1.0044$ & $-1.00287$ & $-1.00085$\\
\hline
$\alpha=-0.93$, $\beta=-1.9$ & $-1.005$& $-1.00327$ & $-1.00097$\\
\hline
$\alpha=-0.93$, $\beta=-1.05$& $-1.0004636$ & $-1.0004635$ & $-1.0045$ \\
\hline
$ \alpha=0.93$, $\beta=-1.05$& $-1.0004094$ & $-1.0004088$ & $-1.0004001$\\
\hline
$\alpha=8.3$, $\beta=-1.55$& $-1.00235$ &$-1.00231$ & $-1.00151$\\
\hline
\end{tabular}
\end{table}

\section{Proof of main results}
We will make use of the following mixed three term recurrence relations satisf\/ied by Jacobi polynomials.
The relations are derived from contiguous relations satisf\/ied by~$_2F_1$ hypergeometric functions and can easily be verif\/ied by comparing coef\/f\/icients of powers of~$x$ on both sides of the equations.

\begin{lemma}Let $\big\{P_n^{(\alpha,\beta)}\big\}_{n=0}^\infty$, $n\in\mathbb{N}$ be the sequence of Jacobi polynomials:
\begin{gather}
2 n (\alpha +\beta +n) P_n^{(\alpha ,\beta )}= -(1+x) (\alpha +n-1) (\alpha +\beta +2 n) P_{n-2}^{(\alpha ,\beta +1)} \nonumber\\
 \qquad{}-[2 (\beta +n) (\alpha +\beta +n)- (x+1) (\alpha +\beta +2 n-1) (\alpha +\beta +2 n)] P_{n-1}^{(\alpha ,\beta )},
\label{2.17}\\
(x+1) (\alpha +\beta +n+1) P_{n-1}^{(\alpha ,\beta +2)}= 2 n P_n^{(\alpha ,\beta )} +2 (\beta +1) P_{n-1}^{(\alpha ,\beta +1)},\label{fo}\\
\label{n2b2}
\frac{(\beta +n)}{2n} \left(x+1-A_n\right) P_{n-1}^{(\alpha ,\beta )}=
\frac{(x+1)^2 (\alpha +n-1)}{4n} P_{n-2}^{(\alpha ,\beta +2)}+\frac{(\beta +1)}{\alpha +\beta +2 n} P_{n}^{(\alpha ,\beta )},\\
\label{n2b3}
 \left(x+B_n\right) P_{n-1}^{(\alpha ,\beta )}-A(x)P_n^{(\alpha ,\beta )}+ \frac{(x+1)^3 (\alpha +n-1) (\alpha +\beta +2 n)}{4 (\beta +n) (\beta +n+1)}P_{n-2}^{(\alpha ,\beta +3)},
\\
\label{n2b4}
 \left(C_n(x+1)-D_n\right)P_{n-1}^{(\alpha,\beta)}
=\frac{(x+1)^4 E_{n}}{8(n+\beta)(\beta+2)}P_{n-2}^{(\alpha,\beta+4)}-\frac{n B(x)}{2(n+\beta)(\beta+2)}P_{n}^{(\alpha,\beta)},
\end{gather}
where $A_n$, $B_n$, $C_n$ and $D_n$ are given in~\eqref{labels},
\begin{gather*}
E_n=(2n+\alpha+\beta)(n+\alpha-1)(n+\alpha+\beta+1)(n+\alpha+\beta+2)
\end{gather*} and
\begin{gather*}A(x)= \frac{n (2 (\beta +1) (\beta +2)-(n-1) (x+1) (\alpha +n-1))}{(\beta +n) (\beta +n+1) (\alpha +\beta +n+1)},\\
B(x)= \big(\alpha ^2+5 \alpha \beta +7 \alpha +4 \beta ^3+24 \beta ^2+39 \beta -2 n^3-3 \alpha n^2-5 \beta n^2-4 n^2-\alpha ^2 n-5 \alpha \beta n \\
\hphantom{B(x)=}{}
-4 \alpha n+10 \beta n+14 n+16\big)-2 (n-1)(n+\alpha-1) (2 n+\alpha +3 \beta+4)x\\
\hphantom{B(x)=}{} -(n-1)(n+\alpha-1)(2n+\alpha +\beta)x^2.
\end{gather*}
\end{lemma}

\begin{proof}[Proof of Theorem \ref{Th:2.1}] Evaluating the mixed three-term recurrence relation \cite[p.~265]{Rai}
\begin{gather*}
\frac12(2+\alpha +\beta +2n) (x+1) P_{n}^{(\alpha,\beta +1)}(x) = (n +1) P_{n +1}^{(\alpha,\beta)}(x) + (1+\beta+n)
P_{n}^{(\alpha,\beta)}(x)
\end{gather*}
at successive zeros $x_{i,n}$, $x_{i+1,n}$
of $P_{n}^{(\alpha,\beta)}$, $i\in\{1,\dots,n-1\}$, we obtain
\begin{gather}4{(n+1)}^2 P_{n +1}^{(\alpha,\beta)} (x_{i,n}) P_{n +1}^{(\alpha,\beta)}(x_{i+1,n})\nonumber\\
\qquad{} = {(2+\alpha +\beta +2n)}^2(x_{i,n}+1)(x_{i+1,n}+1) P_{n}^{(\alpha,\beta +1)}(x_{i,n}) P_{n}^{(\alpha,\beta +1)}(x_{i+1,n}). \label{2.5}
\end{gather}
Now, from \eqref{2.1}, $(1+x_{i,n})( 1+x_{i+1,n})<0$ when $i=1$ and $(1+x_{i,n})(
1+x_{i+1,n})>0$ for $i\in\{2,\dots,n-1\}$ while $P_{n}^{(\alpha,\beta +1)}(x_{i,n})
P_{n}^{(\alpha,\beta +1)}(x_{i+1,n}) <0$ for $i\in\{1,2,\dots,n-1\}$. We deduce from~\eqref{2.5}
that $P_{n +1}^{(\alpha,\beta)}(x_{i,n})$ and $P_{n +1}^{(\alpha,\beta)}(x_{i+1,n})$ have the same
sign for $i=1$ and dif\/fer in sign for $i=2,\dots,n-1$. Since the zeros are distinct, it follows that $P_{n +1}^{(\alpha,\beta)}$
has an even number of zeros in the interval $(x_{1,n},x_{2,n})$ and an odd number of zeros in each of the intervals $(x_{i,n},x_{i+1,n})$, $i\in\{2,\dots,n-1\}$. Therefore $P_{n
+1}^{(\alpha,\beta)}$ has at least $n-2$ simple zeros between $x_{2,n}$ and $x_{n,n}$, plus
its smallest zero $x_{1, n+1}$ which is $<-1$ and, from~\eqref{2.1} and~\eqref{2.2}, its largest zero $x_{n+1,
n+1} > y_{n,n} > x_{n,n}$. Therefore, $n$ zeros of $P_{n+1}^{(\alpha,\beta)}$ are accounted for
and we must still have either no zeros or two zeros of $P_{n +1}^{(\alpha,\beta)}$ in the interval
$(x_{1,n},x_{2,n})$ where $x_{1,n} < -1 < x_{2,n}$ for each $n \in{\mathbb N}$, $n \geq 1$. Since
exactly one of the zeros of $P_{n +1}^{(\alpha,\beta)}$ is $<-1$ for $n\in{\mathbb N}$, $n\geq
1$, the only possibility is $x_{1,n}< x_{1,n+1}< -1 < x_{2,n+1}< x_{2,n}$ which proves the
result.
\end{proof}

\begin{proof}[Proof of Theorem \ref{Th:3.1}]
It follows from \eqref{1} and the assumption that $P_{n}^{(\alpha,\beta)}$
and $P_{n-2}^{(\alpha,\beta)}$ are co-prime that $P_{n}^{(\alpha,\beta)}(d_n)\neq 0$ since $c_n$, $e_n>0$ provided that $n\geq3$. Evaluating~\eqref{1} at the $n-2$ pairs
of successive zeros $x_{i,n}$ and $x_{i+1,n}$, $i\in\{2,\dots,n-1\}$, of $P_{n}^{(\alpha,\beta)}$ that lie in the interval $(-1,1)$, we obtain
\begin{gather}
\frac{P_{n-1}^{(\alpha,\beta)}(x_{i,n})P_{n-1}^{(\alpha,\beta)}(x_{i+1,n})}{P_{n-2}^{(\alpha,\beta)}(x_{i,n})
P_{n-2}^{(\alpha,\beta)}(x_{i+1,n})}= \frac{(e_n)^2}{(d_n-x_{i,n})(d_n-x_{i+1,n})}.\label{4.1}
\end{gather}
The right-hand side of~\eqref{4.1} is positive if and only if $d_n\notin(x_{i,n},x_{i+1,n})$, while
\begin{gather*}
P_{n-1}^{(\alpha,\beta)}(x_{i,n})P_{n-1}^{(\alpha,\beta)}(x_{i+1,n})<0
\end{gather*}
for each $i\in\{2,\dots,n-1\}$ since we know from Theorem~\ref{Th:2.1} that the zeros of $P_{n}^{(\alpha,\beta)}$
and $P_{n-1}^{(\alpha,\beta)}$ that lie in the interval $(-1,1)$ are interlacing. Therefore,
from~\eqref{4.1}, $P_{n-2}^{(\alpha,\beta)}$ changes sign between each pair of successive zeros
of $P_n^{(\alpha,\beta)}$ that lie in $(-1,1)$ except possibly for one pair~$x_{j,n}$,~$x_{j+1,n}$, with $x_{j,n} < d_n < x_{j+1,n}$, $j\in\{2,\dots,n-1\}$. There are
$n-2$ intervals with endpoints at the successive zeros of $P_{n}^{(\alpha,\beta)}$ that lie in
the interval $(-1,1)$ and $P_{n-2}^{(\alpha,\beta)}$ has exactly $n-3$ distinct zeros in
$(-1,1)$. Therefore, the zeros of $P_{n-2}^{(\alpha,\beta)}$ that lie in the interval $(-1,1)$,
together with the point $d_n$, must interlace with the $n-1$ zeros of
$P_n^{(\alpha,\beta)}$ that lie in $(-1,1)$. The stated interlacing result follows from Theorem~\ref{Th:2.1} since $x_{1,n-2} < x_{1,n} < -1 < x_{2,n}$.
\end{proof}

\begin{proof}[Proof of Theorem \ref{Th:Askey1}] Suppose that $\delta < x_{2,n}$. From \cite[equation~(11), p.~71]{Rai},
\begin{gather}
(2(\beta+1) + (x+1)(\alpha +\beta+2n +2)) P_{n}^{(\alpha,\beta+1)} \nonumber\\
\qquad{}= (x+1)(\alpha +\beta+n+2)P_{n}^{(\alpha,\beta+2)}+ 2(\beta +n +1) P_{n}^{(\alpha,\beta)}.\label{5.1}
\end{gather}
Since $P_{n}^{(\alpha,\beta)}$ and $P_{n}^{(\alpha,\beta+1)}$ are co-prime for each $n \in{\mathbb N}$ and each f\/ixed $\alpha$, $\beta$, $\alpha > -1$, $-2 < \beta< -1$ by \eqref{2.1}, it follows from~\eqref{5.1} that the only possible common zero of $P_{n}^{(\alpha,\beta)}$ and $P_{n}^{(\alpha,\beta+2)}$ is $\delta:= -1-\frac{2(\beta+1)}{\alpha +\beta+2n +2}$. If $\delta < x_{2,n}$ then $P_{n}^{(\alpha,\beta)}$ and $P_{n}^{(\alpha,\beta+2)}$ are co-prime
since all the zeros of $P_{n}^{(\alpha,\beta+2)}$ lie in $(-1,1)$ and $x_{2,n}$ is the smallest zero of $P_{n}^{(\alpha,\beta)}$ in $(-1,1)$.
Evaluating~\eqref{5.1} at successive zeros $x_{1,n}< -1< x_{2,n}<\dots<x_{n,n}<1$ of $P_{n}^{(\alpha,\beta)}$ we obtain, for each $i \in \{1,2,\dots,n-1\}$,
\begin{gather} (x_{i,n}+1) (x_{i+1,n}+1)P_n^{(\alpha,\beta+2)}(x_{i,n}) P_n^{(\alpha,\beta +2)}(x_{i+1,n}){(\alpha+\beta +n+2)}^2\nonumber\\
\qquad{}= {(\alpha+\beta +2n+2)}^2(x_{i,n} -\delta) (x_{i+1,n} -\delta)
P_{n}^{(\alpha,\beta+1)}(x_{i,n}) P_{n}^{(\alpha,\beta+1)}(x_{i+1,n}). \label{5.2}
\end{gather}
Now, from \eqref{2.1}, $(x_{i,n}+1)(x_{i+1,n}+1)<0$ when $i=1;$ $(x_{i,n}+1)(x_{i+1,n}+1)>0$ for $i = 2,3,\dots,n-1$ and
$P_n^{(\alpha,\beta+1)}(x_{i,n})P_n^{(\alpha,\beta+1)}(x_{i+1,n}) <0$ for each $i = 1, 2,\dots,n-1$. Since $x_{1,n}<-1<\delta < x_{2,n}$ by assumption, we deduce from~\eqref{5.2} that $P_n^{(\alpha,\beta +2)}(x_{i,n})$ and $P_n^{(\alpha,\beta+2)}(x_{i+1,n})$ dif\/fer in sign for
each $i=1,2,\dots,n-1$. It follows that $P_n^{(\alpha,\beta +2)}$ has an odd number of zeros in each one of the intervals $(x_{i,n},x_{i+1,n})$ for $i = 1, 2,3,\dots,n-1$. Also, from~\eqref{2.1}, $x_{n,n}< z_{n,n}$ where $z_{1,n}<z_{2,n}<\dots <z_{n,n}$ are the zeros of
$P_{n}^{(\alpha,\beta+2)}$. It follows that the zeros of $P_n^{(\alpha,\beta)}$ and $P_n^{(\alpha,\beta+2)}$ are interlacing.

Suppose that $\delta > x_{2,n}$. A similar analysis of~\eqref{5.2} shows that $P_n^{(\alpha,\beta +2)}$ has the same sign at the smallest two zeros $x_{1,n}$ and $x_{2,n}$ of $P_n^{(\alpha,\beta)}$ and therefore an even number of zeros in the interval $(x_{1,n},x_{2,n})$, which shows that interlacing does not hold. Obviously, if $\delta = x_{2,n}$ then~$\delta$ is a~common zero of $P_n^{(\alpha,\beta)}$ and $P_n^{(\alpha,\beta+2)}$ so interlacing does not hold. This completes the proof.
\end{proof}

\begin{proof}[Proof of Theorem \ref{Th:Askey2}]
Evaluating \eqref{5.1} at successive zeros $z_{i,n}$, $z_{i+1,n}$ of $P_{n}^{(\alpha,\beta+2)}$, we have, for each $i \in \{1,2,\dots,n-1\}$,
\begin{gather*}4{(\beta+n+1)} ^2 P_n^{(\alpha,\beta)}(z_{i,n}) P_n^{(\alpha,\beta)}(z_{i+1,n})\\
 \qquad{} = {(\alpha +\beta +2n +2)}^2(z_{i,n} -\delta) (z_{i+1,n} -\delta)
P_{n}^{(\alpha,\beta+1)}(z_{i,n}) P_{n}^{(\alpha,\beta+1)}(z_{i+1,n}).
\end{gather*}
From \cite[Theorem~2.4]{DrJoMb} we know that if $y_{1,n}<y_{2,n}<\dots<y_{n,n}$ denote the zeros of $P_n^{(\alpha,\beta+1)}$, then
\begin{gather}
-1 < y_{1,n} < z_{1,n}< y_{2,n} < z_{2,n} < \dots< y_{n,n}< z_{n,n}<1, \label{4.5}
\end{gather}
so that $P_n^{(\alpha,\beta+1)}(z_{i,n})P_n^{(\alpha,\beta+1)}(z_{i+1,n}) <0$ for each $i = 1, 2,\dots,n-1$, while $(z_{i,n} -\delta)(z_{i+1,n} -\delta) > 0$ unless $\delta \in (z_{i,n}, z_{i+1,n})$. This means that there are two possibilities: (a) $P_n^{(\alpha,\beta)}$ has
$n-1$ sign changes between successive zeros of $P_n^{(\alpha,\beta+2)}$ in $(-1,1)$ and $\delta \notin (z_{i,n}, z_{i+1,n})$ for any $i \in \{1,2,\dots,n-1\}$; or~(b) $P_n^{(\alpha,\beta)}$ has $n-2$ sign changes between successive zeros of $P_n^{(\alpha,\beta+2)}$ in $(-1,1)$ and
$\delta$ lies in one interval, say $\delta \in (z_{j,n}, z_{j+1,n})$ where $j \in \{1,2,\dots,n-1\}$. If~(a) holds then since $P_n^{(\alpha,\beta)}$ has exactly $n-1$ simple zeros in $(-1,1)$, these zeros, together with the point~$\delta$, interlace with the zeros of
$P_n^{(\alpha,\beta+2)}$ in~$(-1,1)$. If, on the other hand, (b)~holds then $P_n^{(\alpha,\beta)}$ has no sign change, and therefore an even number of zeros, in the interval $(z_{j,n}, z_{j+1,n})$ that contains~$\delta$. Since~$P_n^{(\alpha,\beta)}$ has exactly $n-1$ simple zeros in $(-1,1)$ and $n-2$
sign changes in $(-1,1)$, we deduce that no zero of $P_n^{(\alpha,\beta)}$ lies in the interval $(z_{j,n}, z_{j+1,n})$ that
contains~$\delta$ and one zero of~$P_n^{(\alpha,\beta)}$ is either $< z_{1,n}$ or $> z_{n,n}$. Since we know from \eqref{2.1} that
the largest zero $x_{n,n}$ of~$P_n^{(\alpha,\beta)}$ satisf\/ies $x_{n,n}< y_{n,n}$ while from~\eqref{4.5} $y_{n,n}< z_{n,n}$, the only possibility is that the smallest zero~$x_{2,n}$ of~$P_n^{(\alpha,\beta)}$ in $(-1,1)$ is $< z_{1,n}$. Therefore, the zeros of $(x-\delta)P_n^{(\alpha,\beta)}(x)$ interlace with the zeros of~$P_n^{(\alpha,\beta+2)}(x)$ and the result follows.
\end{proof}

\begin{proof}[Proof of Theorem~\ref{Th:la}] Evaluating \eqref{fo} at successive zeros $x_{i,n}$, $x_{i+1,n}$, $i\in\{1,2,\dots,n-1\}$ of~$P_n^{(\alpha,\beta}$ we obtain
\begin{gather*}
(x_{i,n}+1)(x_{i+1,n}+1)(\alpha +\beta +n+1)^2 P_{n-1}^{(\alpha ,\beta +2)}(x_{i,n})P_{n-1}^{(\alpha ,\beta +2)}(x_{i+1,n})\\
\qquad{} = 4 (\beta +1)^2 P_{n-1}^{(\alpha ,\beta +1)}(x_{i,n})P_{n-1}^{(\alpha ,\beta +1)}(x_{i+1,n}).
\end{gather*}
From \eqref{2.2} with $n$ replaced by $n-1$ we have $P_{n-1}^{(\alpha ,\beta +1)}(x_{i,n})P_{n-1}^{(\alpha ,\beta +1)}(x_{i+1,n})<0$ while $(1+x_{i,n})(1+x_{i+1,n})<0$ for $i=1$ and $>0$ for $i\in\{2,3,\dots,n-1\}$. Therefore $P_{n-1}^{(\alpha ,\beta +2)}(x_{i,n})P_{n-1}^{(\alpha ,\beta +2)}(x_{i+1,n}){<}0$ for each $i\in\{2,3,\dots,n-1\}$. Hence $P_{n-1}^{(\alpha ,\beta +2)}$ has an even number of sign changes in $(x_{1,n},x_{2,n})$ and an odd number of sign changes in $(x_{i,n},x_{i+1,n})$ for $i\in\{2,3\dots,n-1\}$. Since $P_{n-1}^{(\alpha,\beta+2)}$ has $n-1$ distinct zeros, there must be exactly one zero of $P_{n-1}^{(\alpha,\beta+2)}$ in each of the $n-2$ intervals $(x_{i,n},x_{i+1,n})$, $i\in\{2,3\dots,n-1\}$. The remaining zero of $P_{n-1}^{(\alpha,\beta+2)}$ must lie in $(-1,1)$ and cannot lie in the interval $(-1,x_{2,n})$. Therefore the only possibility is that the largest zero $z_{n-1,n-1}$ of $P_{n-1}^{(\alpha,\beta+2)}$ is $>x_{n,n}$.
\end{proof}
\begin{proof}[Proof of Theorem \ref{Th:St1}]
(i)~We can write \eqref{2.17} as \begin{gather}\label{2.18}
 (k_1 - (x+1)k_2)P_{n-1}^{(\alpha,\beta)} (x) = -(1+x)k_3P_{n-2}^{(\alpha,\beta+1)}(x) - k_4 P_{n}^{(\alpha,\beta)} (x).
\end{gather}
Let $x_{i,n}$, $i\in\{1,2,\dots,n\}$ denote the zeros of $P_{n}^{(\alpha,\beta)}$ in ascending order. Note that $(1+x_{i,n}) \neq \frac{k_1}{k_2}$ for any $i\in\{1,\dots,n\}$ since that would
contradict the assumption that $P_{n}^{(\alpha,\beta)}$ and $P_{n-2}^{(\alpha,\beta+1)}$ are
co-prime. Evaluating~\eqref{2.18} at each pair of zeros $x_{i,n}$ and $x_{i+1,n}$,
$i\in\{2,\dots,n-1\}$, of $P_{n}^{(\alpha,\beta)}$ that lie in the interval $(-1,1)$, we obtain
\begin{gather} \label{2.19}
\frac{P_{n-1}^{(\alpha,\beta)}(x_{i,n})P_{n-1}^{(\alpha,\beta)}(x_{i+1,n})}{P_{n-2}^{(\alpha,\beta+1)}(x_{i,n})
P_{n-2}^{(\alpha,\beta+1)}(x_{i+1,n})}= \frac{(1+x_{i,n})(1+ x_{i+1,n}) k_3^2}{(k_1-
(x_{i,n}+1)k_2)(k_1- (x_{i+1,n}+1)k_2)}.
\end{gather}
Since $(1+x_{i,n})$ and $(1+x_{i+1,n})$ are positive for $i\in\{2,\dots,n-1\}$, the right-hand
side of \eqref{2.19} is positive if and only if $\frac{k_1}{k_2} -1 \notin(x_{i,n},x_{i+1,n})$
for any $i\in\{2,\dots,n-1\}$. Suppose, now, that
$\frac{k_1}{k_2} -1\notin(x_{i,n},x_{i+1,n})$ for any $i\in\{2,\dots,n-1\}$. Since the zeros $x_{i,n-1}$, $i\in\{2,\dots,n-1\}$ of~$P_{n-1}^{(\alpha,\beta)}$ interlace with the
zeros $x_{i,n}$, $i\in\{2,\dots,n\}$ of $P_n^{(\alpha,\beta)}$, $\alpha > -1$, $-2<\beta<-1$ (Theorem~\ref{Th:2.1}),
we see from~\eqref{2.19} that $P_{n-1}^{(\alpha,\beta)}(x_{i,n})P_{n-1}^{(\alpha,\beta)}(x_{i+1,n}) < 0$ for each
$i\in\{2,\dots,n-1\}$, $n \in{\mathbb N}$, $n \geq 2$. Therefore if
$\frac{k_1}{k_2} -1\notin(x_{i,n},x_{i+1,n})$ for any $i\in\{2,\dots,n-1\}$, the $n-2$ distinct
zeros of~$P_{n-2}^{(\alpha,\beta+1)}$ in $(-1,1)$ interlace with the $n-1$ zeros of~$P_{n}^{(\alpha,\beta)}$ that lie in $(-1,1)$. Further, by our assumption, the point
$\frac{k_1}{k_2} -1$ lies outside the interval with endpoints at the smallest positive zero
$x_{2,n}$ of~$P_{n}^{(\alpha,\beta)}$ and its largest zero $x_{n,n}$ so interlacing holds between the $n-2$ simple zeros of~$ P_{n-2}^{(\alpha,\beta+1)}$ together with
the point $\frac{k_1}{k_2} -1$ and the $n-1$ zeros of~$P_n^{(\alpha,\beta)}$ that lie in
$(-1,1)$. Suppose now that $\frac{k_1}{k_2} -1\in(x_{i,n}, x_{i+1,n})$ for some
$i\in\{2,\dots,n-1\}$. Then, in this single interval say $(x_{j,n}, x_{j+1,n})$ containing
$\frac{k_1}{k_2} -1$, there will be no sign change of~$P_{n-1}^{(\alpha,\beta)}$ but its sign
will change in each of the remaining $n-3$ intervals with endpoints at the successive zeros of
$P_{n}^{(\alpha,\beta)}$. However, evaluating~\eqref{2.18} at~$x_{1,n}$ and~$x_{2,n}$, we obtain
\begin{gather}\label{2.20}
\frac{P_{n-1}^{(\alpha,\beta)}(x_{1,n})P_{n-1}^{(\alpha,\beta)}(x_{2,n})}{P_{n-2}^{(\alpha,\beta+1)}(x_{1,n})
P_{n-2}^{(\alpha,\beta+1)}(x_{2,n})}= \frac{(1+x_{1,n})(1+ x_{2,n})k_2^2 k_3^2}{(\frac{k_1}{k_2} -1-x_{1,n})(\frac{k_1}{k_2} -1-x_{2,n})}.
\end{gather}
Now $\frac{k_1}{k_2} -1\in(x_{i,n}, x_{i+1,n})$ for some $i\in\{2,\dots,n-1\}$ so $\frac{k_1}{k_2} -1\notin(x_{1,n}, x_{2,n})$. The right-hand side of
\eqref{2.20} is therefore negative while, from Theorem \ref{Th:2.1} with $n$ replaced by $n-1$, we know that
$P_{n-1}^{(\alpha,\beta)}(x_{1,n})P_{n-1}^{(\alpha,\beta)}(x_{2,n})>0$.
Therefore $P_{n-2}^{(\alpha,\beta+1)}$ has a dif\/ferent sign at $x_{1,n}$ and~$x_{2,n}$ and
therefore one zero greater that $-1$ but less than $x_{2,n}$. We can therefore deduce that in each case, the $n-2$
simple zeros of $P_{n-2}^{(\alpha,\beta+1)}$, together with the point $\frac{k_1}{k_2} -1$,
interlace with the $n-1$ zeros of $P_n^{(\alpha,\beta)}$ in $(-1,1)$ if
$P_{n-2}^{(\alpha,\beta+1)}$ and $P_n^{(\alpha,\beta)}$ are co-prime.

(ii) Since the zeros of $P_{n-2}^{(\alpha,\beta +t)}$ are increasing functions of $t$ for
$2 \le t \le 4$ \cite[Theorem~6.21.1]{Sze}, it will be suf\/f\/icient to prove~(ii) in the two
special cases $t=2$ and $t=4$.
For the case $t=2$, we note that since the polynomials $P_{n}^{(\alpha,\beta)}$ and $P_{n-1}^{(\alpha,\beta)}$ are
co-prime~\eqref{2.3}, it follows from \eqref{n2b2} that the only possible common zero of
$P_{n-2}^{(\alpha,\beta +2)}$ and $P_{n}^{(\alpha,\beta)}$ is $-1 + \frac{2(\beta+1)}{\alpha +
\beta + 2n}$ which is $<-1$ for each $\alpha$, $\beta$, $\alpha>-1$, $-2<\beta < -1$. Since all of
the zeros of $P_{n-2}^{(\alpha,\beta +2)}$ lie in the interval $(-1,1)$, $P_{n-2}^{(\alpha,\beta +2)}$ and
$P_{n}^{(\alpha,\beta)}$ are co-prime for $\alpha>-1$, $-2<\beta < -1$.
Evaluating \eqref{n2b2} at the $n-2$ pairs of successive zeros $x_{i,n}$ and $x_{i+1,n}$,
$i\in\{2,\dots,n-1\}$ of $P_{n}^{(\alpha,\beta)}$ that lie in $(-1,1)$ yields
\begin{gather}
 \frac{P_{n-1}^{(\alpha,\beta)}(x_{i,n})P_{n-1}^{(\alpha,\beta)}(x_{i+1,n})}{P_{n-2}^{(\alpha,\beta+2)}(x_{i,n})
P_{n-2}^{(\alpha, \beta+2)}(x_{i+1,n})} \nonumber\\
\qquad{} = \frac{(x_{i,n}+1)^2 (x_{i+1,n}+1)^2(\alpha+n-1)^2(\alpha+\beta+2n)^2}
{4(\beta+n)^2 (2n+\alpha+\beta)^2(x_{i,n}+1-A_n)(x_{i+1,n}+1-A_n)}.\label{2.21}
\end{gather}
The right-hand side of \eqref{2.21} is positive since $\frac{2(\beta+1)}{\alpha +\beta+2n} \notin (1 + x_{i,n}, 1 + x_{i+1,n})$ for any $i\in\{2,\dots$, $n-1\}$.
By Theorem \ref{Th:2.1}, $P_{n-1}^{(\alpha,\beta)}(x_{i,n})P_{n-1}^{(\alpha,\beta)}(x_{i+1,n} ){<} 0$ and
hence $P_{n-2}^{(\alpha,\beta+2)}(x_{i,n}) P_{n-2}^{(\alpha,\beta+2)}(x_{i+1,n}) {<} 0$ for each
$i\in\{2,\dots,n-1\}$, and, for $t=2$, the interlacing result follows.

For the case $t=4$, since $P_{n}^{(\alpha,\beta)}$ and $P_{n-1}^{(\alpha,\beta)}$ are
co-prime, \eqref{n2b4} implies that the only possible common zero of
$P_{n-2}^{(\alpha,\beta +4)}$ and $P_{n}^{(\alpha,\beta)}$ is $-1+\frac{D_n}{C_n}$. Since $D_n<0$ and $C_n>0$ for each
$\alpha$, $\beta$ with $\alpha>-1$, $-2<\beta < -1$, and $n \geq 3$, it follows that $-1+\frac{D_n}{C_n}<-1$ and therefore, since all of the zeros of
$P_{n-2}^{(\alpha,\beta +4)}$ lie in $(-1,1)$, $P_{n-2}^{(\alpha,\beta +4)}$ and $P_{n}^{(\alpha,\beta)}$ are co-prime for $\alpha>-1$, $-2<\beta < -1$.
Evaluating \eqref{n2b4} at the $n-2$ pairs of successive zeros $x_{i,n}$ and $x_{i+1,n}$,
$i\in\{2,\dots,n-1\}$ of $P_{n}^{(\alpha,\beta)}$ that lie in $(-1,1)$,
\begin{gather}
 \frac{P_{n-1}^{(\alpha,\beta)}(x_{i,n})P_{n-1}^{(\alpha,\beta)}(
x_{i+1,n})}{P_{n-2}^{(\alpha,\beta+4)}(x_{i,n})
P_{n-2}^{(\alpha, \beta+4)}(x_{i+1,n})}\nonumber\\
\qquad{}
=\frac{(x_{i,n}+1)^4(x_{i+1,n}+1)^4E_n^2}{64(n+\beta)^2(\beta+2)^2(C_n(x_{i,n}+1)+D_n)(C_n(x_{i+1,n}+1)+D_n)}.\label{n2b4r}
\end{gather}
The right-hand side of \eqref{n2b4r} is positive since $\frac{D_n}{C_n} \notin (1 + x_{i,n}, 1 + x_{i+1,n})$ for $i\in\{2,\dots,n-1\}$. By Theorem \ref{Th:2.1}, $P_{n-1}^{(\alpha,\beta)}(x_{i,n})P_{n-1}^{(\alpha,\beta)}(x_{i+1,n} )< 0$ and
hence $P_{n-2}^{(\alpha,\beta+4)}(x_{i,n}) P_{n-2}^{(\alpha,\beta+4)}(x_{i+1,n}) < 0$ for each
$i\in\{2,\dots,n-1\}$, and, for $t=4$, the interlacing result follows.
\end{proof}

\begin{proof}[Proof of Theorem \ref{Th:bounds}] Let $x_{1,n}$ and $y_{1,n-2}^{(\alpha,\beta+t)}$, $t\in\{2,3,4\}$ denote the smallest zero of $P_n^{(\alpha,\beta)}$ and $P_{n-2}^{(\alpha,\beta+t)}$ respectively. It follows from~\eqref{2.3}, Theorem~\ref{Th:St1}(ii) and the monotonicity of the zeros of Jacobi polynomials (cf.\ \cite[Theorem~7.1.2]{Ism}) that
\begin{gather}\label{order}
x_{1,n-1}<x_{1,n}<-1<x_{2,n}<y_{1,n-2}^{(\alpha,\beta+2)}<y_{1,n-2}^{(\alpha,\beta+3)}<y_{1,n-2}^{(\alpha,\beta+4)}.
\end{gather}
Since $\lim\limits_{x\to -\infty}P_n^{(\alpha,\beta)}(x)=\infty$ for $n$ even, while $\lim\limits_{x\to -\infty}P_n^{(\alpha,\beta)}(x)=-\infty$ for~$n$ odd, we deduce from~\eqref{order} that
\begin{gather}\label{sign}
\frac{P_{n-2}^{(\alpha,\beta+t)}(x_{1,n})}{P_{n-1}^{(\alpha,\beta)}(x_{1,n})}>0 \qquad \text{for} \quad t\in\{2,3,4\}.
\end{gather}
Evaluating \eqref{n2b2} at $x_{1,n}$, we obtain
\begin{gather*}
\frac{P_{n-2}^{(\alpha,\beta+2)}(x_{1,n})}{P_{n-1}^{(\alpha,\beta)}(x_{1,n})}=
\frac{2(\beta+n)((x_{1,n}+1)(2n+\alpha+\beta)-2(\beta+1))}{(x_{1,n}+1)^2(n+\alpha-1)(2n+\alpha+\beta)}
\end{gather*}
and therefore it follows from \eqref{sign} that $(x_{1,n}+1)(2n+\alpha+\beta)-2(\beta+1)>0$ for $n\geq3$.
This yields the bound \begin{gather*}
x_{1,n}>\frac{2(\beta+1)}{2n+\alpha+\beta}-1.
\end{gather*}

Next, evaluating \eqref{n2b3} at $x_{1,n}$ we obtain
\begin{gather}\label{ag}
\frac{P_{n-2}^{(\alpha,\beta+3)}(x_{1,n})}{P_{n-1}^{(\alpha,\beta)}(x_{1,n})}
=\frac{4(n+\beta)(n+\beta+1)(x_{1,n}+B_n)}{(x_{1,n}+1)^3(n+\alpha-1)(2n+\alpha+\beta)}.
\end{gather}
Since the left hand side of \eqref{ag} is positive by~\eqref{sign}, $(x_{1,n}+1)^3<0$ by~\eqref{2.1} and $B_n>1$ for $n\geq 3$, $\alpha>-1$, $-2<\beta<-1$, we see that $x_{1,n}<-B_n<-1$.

Evaluating \eqref{n2b4} at $x_{1,n}$ we obtain{\samepage
\begin{gather*}
\frac{P_{n-2}^{(\alpha,\beta+4)}(x_{1,n})}{P_{n-1}^{(\alpha,\beta)}(x_{1,n})}
=\frac{8(n+\beta)(\beta+2)(C_n(x_{1,n}+1)-D_n)}{(x_{1,n}+1)^4E_n}
\end{gather*}
and it follows from \eqref{sign} that $C_n(x_{1,n}+1)-D_n>0$.}

Finally, since $C_n-(\beta+3)(2n+\alpha+\beta)=2(n-1)(n+\alpha-1)>0$ for $n\geq 3$ and $\alpha>-1$, we see that
\begin{gather*}
-1+\frac{2(\beta+1)}{2n+\alpha+\beta}<-1+\frac{2(\beta+1)(\beta+3)}{C_n}
\end{gather*} for each $n\geq 3$, $\alpha>-1$ and $-2<\beta<-1$.
\end{proof}

\subsection*{Acknowledgments}

The research of both authors was funded by the National Research Foundation of South Africa. We thank the referees for helpful suggestions and insights.

\pdfbookmark[1]{References}{ref}
\LastPageEnding

\end{document}